\documentclass[12pt]{amsart}
\usepackage{amssymb,amscd,amsthm,amsmath,color}
\usepackage{fullpage}
\newcommand{\PP}{\mathbb{P}}
\newcommand{\OO}{\mathcal{O}}

\newcommand{\gG}{\mathbb{G}}

\newcommand{\uU}{\mathcal{U}}
\newcommand{\xX}{\mathcal{X}}
\newcommand{\yY}{\mathcal{Y}}
\newcommand{\sing}{\textsf{sing}}

\newcommand{\ev}{\operatorname{ev}}

\newcommand{\spa}{\operatorname{span}}

\newtheorem{theorem}{Theorem}[section]
\newtheorem{lemma}[theorem]{Lemma}
\newtheorem{proposition}[theorem]{Proposition}
\newtheorem{corollary}[theorem]{Corollary}

\newtheorem{conjecture}[theorem]{Conjecture}
\newtheorem{question}[theorem]{Question}
\newtheorem{remark}[theorem]{Remark}

\begin{document}

\title{Linear subspaces of hypersurfaces}
\author[R. Beheshti]{Roya Beheshti}
\address{Department of Mathematics and Statistics\\Washington University in St. Louis\\1 Brookings Drive\\St. Louis, MO 63130}
\email{beheshti@math.wustl.edu}

\author[E. Riedl]{Eric Riedl}
\address{Department of Mathematics\\University of Notre Dame\\255 Hurley Hall\\Notre Dame, IN 46556}
\email{eriedl@nd.edu}

\begin{abstract}
Let $X$ be an arbitrary smooth hypersurface in $\mathbb{C} \PP^n$ of degree $d$. We prove the de Jong-Debarre Conjecture for $n \geq 2d-4$: the space of lines in $X$ has dimension $2n-d-3$. We also prove an analogous result for $k$-planes: if $n \geq 2 \binom{d+k-1}{k} + k$, then the space of $k$-planes on $X$ will be irreducible of the expected dimension. As applications, we prove that an arbitrary smooth hypersurface satisfying $n \geq 2^{d!}$ is unirational, and we prove that the space of degree $e$ curves on $X$ will be irreducible of the expected dimension provided that $d \leq \frac{e+n}{e+1}$.
\end{abstract}

\maketitle

\section{Introduction}
We work throughout over an algebraically closed field of characteristic 0. Let $X \subset \PP^n$ be an arbitrary smooth hypersurface of degree $d$. Let $F_k(X) \subset \gG(k,n)$ be the Hilbert scheme of $k$-planes contained in $X$.

\begin{question}
\label{q-indDim}
What is the dimension of $F_k(X)$?
\end{question}

In particular, we would like to know if there are triples $(n,d,k)$ for which the answer depends only on $(n,d,k)$ and not on the specific smooth hypersurface $X$. It is known classically that $F_k(X)$ is locally cut out by $\binom{d+k}{k}$ equations. Therefore, one might expect the answer to Question \ref{q-indDim} to be that the dimension is $(k+1)(n-k) - \binom{d+k}{k}$, where negative dimensions mean that $F_k(X)$ is empty. This is indeed the case when the hypersurface $X$ is general \cite{3264, hochster-laksov}. Standard examples (see Proposition \ref{prop-counterexample}) show that $\dim F_k(X)$ must depend on the particular smooth hypersurface $X$ for $d$ large relative to $n$ and $k$, but there remains hope that Question \ref{q-indDim} might be answered positively for $n$ large relative to $d$ and $k$.

In the special case $k=1$ there is a conjectured answer as to when $F_1(X)$ has the expected dimension.

\begin{conjecture}[de Jong-Debarre]
\label{conj-deJongDebarre}
If $X \subset \PP^n$ is a smooth hypersurface with $d \leq n$, then $F_1(X)$ has the expected dimension $2n-d-3$.
\end{conjecture}

The bound $n \geq d$ in Conjecture \ref{conj-deJongDebarre} is known to be optimal (see Proposition \ref{prop-dJBoptimal}). In the case $k > 1$ we know of no conjecture as to what the optimal bound for $n$ in terms of $d$ and $k$ should be, although there are immediate lower bounds coming from Proposition \ref{prop-counterexample}.

The subject has received much interest over several decades. All prior work has either required $n$ to grow at least exponentially with $d$ or has been for finitely many values of $d$. Harris, Mazur, and Pandharipande \cite{HMP} in their study of unirational parameterizations of smooth hypersurfaces answer Question \ref{q-indDim} in the affirmative for $n$ extremely large relative to $d$ (an iterated exponential). However, their bound was not expected to be optimal, and is still at least exponential even in the case $k=1$.

Conjecture \ref{conj-deJongDebarre} is known for small degree. Debarre proved the result for $d \leq 5$ in unpublished work, Beheshti, Landsberg-Tommasi, and Landsberg-Robles prove the conjecture for $d \leq 6$ \cite{beheshtid6, landsberg-tommasi, landsberg-robles}. Beheshti \cite{beheshtid8} proves Conjecture \ref{conj-deJongDebarre} for $d \leq 8$. 

We address Question \ref{q-indDim} for all degrees $k$, and prove a result that is within a factor of $k+1$ of being optimal. We additionally prove irreducibility.

\begin{theorem}[cf Theorem \ref{thm-irrkplanes} and Corollary \ref{cor-irrLines}]
\label{thm-introkplanes}
Let $X \subset \PP^n$ be a smooth hypersurface. Then $F_k(X)$ will be irreducible of the expected dimension provided that
\[  n \geq 2 \binom{d+k-1}{k}+k .\]
In the special case $k=1$, we can improve the bound, proving that $F_1(X)$ is of the expected dimension if $n \geq 2d-4$ and irreducible if $n \geq 2d-1$ and $n \geq 4$.
\end{theorem}

For comparison, Beheshti and Starr \cite{beheshtikplanes} conjecture that if $d > n-2k+1$, then $F_k(X)$ has dimension at most the dimension of $k$-planes in a maximum dimensional linear space in $X$.

To our knowledge, Theorem \ref{thm-introkplanes} is the first result on the de Jong-Debarre Conjecture that works for all degrees $d$ and does not require $n$ to grow exponentially with $d$. The technique relies on a new result that says, essentially, smooth high degree hypersurfaces tend not to be tangent to varieties cut out by lower-degree equations (see Lemma \ref{lem-polys}). The approach is somewhat similar in philosophy, although not in technique, to results of Ananyan, Hochster, Erman, Sam, and Snowden \cite{ananyan-hochster, erman-sam-snowden} in that it describes how a fixed number of smooth equations tend to become algebraically independent as the number of variables grows.

Results like these, in addition to being of interest in their own right, are significant for the birational and arithmetic geometry of hypersurfaces. For instance, we apply our improved bounds on the space of $k$-planes to Harris, Mazur and Pandharipande's result about unirationality of arbitrary smooth hypersurfaces. 

\begin{theorem}[cf Corollary \ref{cor-unirationalityBound}] \label{thm-unirationalityIntro}
An arbitrary smooth degree $d$ hypersurface in $\PP^n$ is unirational provided that $n \geq 2^{d!}.$
\end{theorem}

Unirational varieties have dense rational points, so proving unirationality is a way to prove that a variety has dense rational points over some number field. It is an important open question whether or not every rationally connected variety is unirational. Smooth hypersurfaces with $d \leq n$ are known to be rationally connected. It is generally expected that Fano hypersurfaces with $d$ approximately equal to $n \geq 5$ should not be unirational, but not even a single example has been proven. 

Theorem \ref{thm-unirationalityIntro} also fits into a long and rich history of the study of rationality properties of hypersurfaces. Quadrics of any dimension and cubic surfaces are classically known to be rational. Any cubic hypersurface of dimension at least two is known to be unirational, but rationality is more subtle. A celebrated result of Clemens and Griffiths \cite{clemens-griffiths} proves that smooth cubic threefolds are unirational but not rational. Other work \cite{kollarHyp, totaro, HMP, schreieder} has culminated in Schreieder's result that a very general hypersurface is not (stably) rational when $n < 2^d$.  
Taken together, our result and Schreieder's suggest that low degree hypersurfaces start to exhibit rationality properties, but the degree $n$ must be quite large with respect to $d$ before they all do so.

We also use Theorem \ref{thm-introkplanes} to prove that the space parametrizing smooth rational curves of small degrees on an arbitrary smooth hypersurface is irreducible. Let $R_e(X)$ be the Hilbert scheme of smooth rational curves of degree $e$ in $X$. In \cite{riedl-yang} Riedl and Yang, building on results of Harris, Roth, and Starr \cite{harris-roth-starr}, show that for $d \leq n-2$ and a general hypersurface $X$ of degree $d$ in $\PP^n$, $R_e(X)$ is irreducible of the expected dimension. More generally they show that the Kontsevich moduli space $\overline{M}_{0,0}(X,e)$ which compactifies $R_e(X)$ is 
irreducible in this degree range. The question which arises is: for fixed $n$, how small must $d$ be in order for the dimension of $R_e(X)$ to equal the expected dimension for all smooth hypersurfaces $X$ of degree $d$ in $\PP^n$? Of course, for $d=1, 2$ and $d \leq n-2$, $R_e(X)$  is always irreducible of the expected dimension. Coskun and Starr \cite{coskun-starr} show that for every smooth cubic hypersurface of dimension at least 4, $\overline{M}_{0,0}(X,e)$  is irreducible of the expected dimension. Lehmann and Tanimoto \cite{lehmann-tanimoto} prove $R_e(X)$ has the expected dimension if $X$ is any smooth quartic hypersurface of dimension at least 5, and relate the question to the $a$-values of subvarieties of hypersurfaces and Manin's Conjecture. Browning and Vishe \cite{browning-vishe} use the circle method to prove that $R_e(X)$ is irreducible of the expected dimension for all $e$ when $n$ is exponentially large compared to $d$. Unfortunately, their technique seems to use the exponential growth of $n$ with $d$ in an essential way, and there is no known way to use the circle method to get a polynomial bound even in the special case $e=1$.

\begin{theorem}[cf Theorem \ref{higherdeg}] \label{thm-higherDegreeIntro}
If $X$ is an arbitrary smooth hypersurface of degree $d$ in $\PP^n$, then $\overline{M}_{0,0}(X,e)$ (and hence, $R_e(X)$) is irreducible of the expected dimension provided $d \leq \frac{e+n}{e+1}$.
\end{theorem}
Our results settle the strong version of geometetric Manin's Conjecture, as outlined by Lehmann and Tanimoto \cite{lehmann-tanimoto}, in the given degree range. Since it applies to the entire Kontsevich space $\overline{M}_{0,0}(X,e)$, Theorem \ref{thm-higherDegreeIntro} proves that the relevant Gromov-Witten invariants of such hypersurfaces are enumerative. For enumerativity, this result is nearly sharp, see the discussion following Proposition \ref{prop-dJBoptimal}. If one is only concerned with irreducible rational curves, it is a drawback that our bounds work only for finitely many $e$ unlike Browning and Vishe's results \cite{browning-vishe}. However, for a fixed $e$, the bound for $n$ is linear in $d$. Thus, we improve on Browning and Vishe's results for small degrees $e$.


\subsection*{Acknowledgements}
We would like to thank Izzet Coskun and Brian Lehmann for helpful conversations. We would also like to thank the anonymous referees for many useful comments. During the work on this paper, Beheshti was partially supported by Simons Collaboration Grant 359447 and Riedl was partially supported by an AMS-Simons Travel Grant.

\subsection*{Outline of Paper} In Section \ref{sec-caseoflines}, we prove the de Jong-Debarre Conjecture for $n$ approximately $2d$. In Section \ref{sec-kplanes}, we generalize these results to $k$-planes, although we have to weaken the bound slightly when working with $k$-planes. In Section \ref{sec-unirationality} we prove unirationality of arbitrary smooth hypersurfaces in sufficiently high degree. In Section \ref{sec-higherdegree} we prove that the spaces $R_e(X)$ have the expected dimension for small $e$.

\section{The case of lines}
\label{sec-caseoflines}
We start by proving that $F_k(X)$ has the expected dimension in the special case $k=1$. Considering lines separately allows us to make a few improvements of the result for general $k$. Additionally, the case of lines is a nice example of the general technique for $k$-planes. While many components of the proof for $k$-planes are the same, the amount of notation somewhat obscures the essential ideas.

\subsection{Dimension}

We start by proving that any component of $F_1(X)$ has the expected dimension. The idea of the proof is as follows. We can write down explicit equations $f_1, \dots, f_d$ for the space of lines $F^p(X)$ passing through a given point $p$ of $X$, and note that $f_d$ will be smooth. We show that for a general point $p$ on a given line $\ell$, the tangent space to $F^p(X)$ at $\ell$ will be locally cut out by the first $\delta$ equations, $f_1, \dots, f_\delta$. We then apply the following lemma (Lemma \ref{lem-polys}) which roughly says that smooth equations of higher degree tend not to be tangent to varieties cut out by equations of lower degree.

\begin{lemma}
\label{lem-polys}
Let $h_1, \dots, h_r$ be homogeneous polynomials on $\PP^n$ of degree strictly less than $d$ and let $h$ be a polynomial of degree $d$ such that $V(h)$ has singular locus of dimension $s$, where $s=-1$ if $V(h)$ is smooth. Then the locus where $V(h)$ is tangent to $V(h_1,\dots,h_r)$ has dimension at most $r+s$.
\end{lemma}
\begin{proof}
Given a polynomial $g$, let $L_p(g)$ be the linear part of $g$ near $p$. Fixing $h_1, \dots, h_r$ and $h$, let $x_0$ be a general homogeneous coordinate on $\PP^n$. Then $V(h)$ will be tangent to $V(h_1, \dots, h_r)$ at $p \in \PP^n \setminus V(x_0)$ if and only if $L_p(h)$ can be written as a linear combination of the $L_p(h_i)$. This will happen if and only if the vanishing locus of the polynomial $h - \sum_i \alpha_i h_i x_0^{d - \deg h_i}$ is singular at $p$ for some tuple of numbers $(\alpha_1, \dots, \alpha_r)$. For each $\alpha = (\alpha_1, \dots, \alpha_r)$, let $g_{\alpha}$ be the polynomial $h - \sum_i \alpha_i h_i x_0^{d - \deg h_i}$.

We start by considering the case where $s \geq 0$. Since $x_0$ is general, we have that $V(h, x_0)=V(g_{\alpha}, x_0)$ is singular in dimension $s-1$. This implies that $V(g_{\alpha})$ has singular locus of dimension at most $s$. Since there is at most an $r$-dimensional family of $\alpha_1, \dots, \alpha_r$, the result follows.

Now suppose $V(h)$ is smooth. Then $V(h,x_0)=V(g_{\alpha}, x_0)$ is also smooth, so that $V(g_{\alpha})$ has at most finitely many singularities. There is an $r$-dimensional family of $\alpha_1, \dots, \alpha_r$. However, the set of $\alpha$ for which $g_{\alpha}$ has singular points is codimension at least $1$ in the space of $\alpha$, since $g_{(0, \dots, 0)} = h$. Thus, the union of the singular loci of the $g_{\alpha}$ will have dimension at most $r-1$.
\end{proof}

We now describe the equations cutting out the space of lines $F^p(X)$ contained in $X$ passing through a point $p$. These results are well-known in the literature, see \cite{harris-roth-starr} for instance. Let $X = V(f) \subset \PP^n$ be a hypersurface of degree $d$. Then given a point $p$ and a homogeneous coordinate $x_0$, we can expand the equation of $f$ around $p$ with respect to $x_0$. If we choose coordinates so that $p = [1, 0, \dots, 0]$, then we can write
\begin{equation}\label{eq-homPiecesOfF} f = \sum_{i=1}^d f_i x_0^{d-i} . \end{equation}
We can view $x_1, \dots, x_n$ as homogeneous coordinates on the space $\PP^{n-1}$ of lines in $\PP^n$ through $p$. We have $F^p(X) = V(f_1, \dots, f_d)$. The \emph{expected dimension} of $F^p(X)$ is $n-1-d$. For $n \geq d+1$, we see that $F^p(X)$ will have the expected dimension precisely when it is a complete intersection.

Our analysis will rest on studying the tangent space to $F^p(X)$ at $\ell$. By standard deformation theory (cf \cite{kollar}, Corollary II.3.10.1), the tangent space to $F^p(X)$ at $\ell$ is given by $H^0(N_{\ell/X}(-p))$, where $N_{\ell/X}$ is the normal bundle of $\ell$ in $X$. If $X$ is smooth, $N_{\ell/X}$ will be a vector bundle. Vector bundles on $\PP^1$ split as a sum of line bundles by Grothendieck's lemma (see \cite{hartshorne}, Ex V.2.4), so we can write $N_{\ell/X} = \bigoplus_{i=1}^{n-2} \OO(a_i)$. Using the short exact sequence
\[ 0 \to N_{\ell/X} \to N_{\ell/ \PP^n} \to N_{X/\PP^n}|_{\ell} \to 0 \]
and the fact that $N_{\ell/ \PP^n} = \OO(1)^{n-1}$ and $N_{X/\PP^n}|_{\ell} = \OO_{\ell}(d)$ we see that $\sum_i a_i = n-d-1$ and that $a_i \leq 1$ for all $i$. From this it follows that $N_{\ell/X}$ will be globally generated precisely when $N_{\ell/X}$ has $n-d-1$ $\OO(1)$ summands and $d-1$ $\OO$ summands. We refer to lines with globally generated $N_{\ell/X}$ as free lines. By standard arguments (see for example Corollary II.3.10.1 from \cite{kollar}), it follows that if $\ell$ is a general line in a family of lines sweeping out $X$, then $N_{\ell/X}$ is globally generated. We state this as a proposition for later use.

\begin{proposition}[cf Corollary II.3.10.1 from \cite{kollar}]
\label{prop-globGenGeneralPoint}
If $X$ is smooth and $S \subset F_1(X)$ is a family of lines that sweep out all of $X$, then a general line in $S$ is free. Thus, if $S$ is a family of non-free lines, then the lines in $S$ must lie in some proper subvariety of $X$.
\end{proposition}

Let $\ell$ be a line in $X$. We show that for a general point $p$ on $\ell$, the tangent space to $F^p(X)$ depends only on the lower-degree equations.

\begin{lemma}
\label{lem-linesDependence}
Let $X = V(f)$ be a hypersurface in $\PP^n$, and let $\ell \in F_1(X)$ be a line. Let $p \in \ell$ be a general point, let $x_0$ be a general coordinate on $\PP^n$, and let $f_1, \dots, f_d$ be the expansion of $f$ around $p$ as described in (\ref{eq-homPiecesOfF}). Then there exists an integer $\delta$ with $0 \leq \delta \leq d$ such that $V(f_1), \dots, V(f_\delta)$ meet transversely at $\ell$, while $T_{\ell} V(f_j)$ contains $T_{\ell} V(f_1, \dots, f_\delta)$ for all $j > \delta$.
\end{lemma}
\begin{proof}
Choose coordinates so that $p$ is $[1,0,\dots,0]$ and $\ell$ is $V(x_2, \dots, x_n)$. Each $f_i$ is a homogeneous polynomial in $x_1, \dots, x_n$, viewed as polynomials on the $\PP^{n-1}$ of lines passing through $p$. Let $L(f_i)$ be the linear part of $f_i$ at the line $\ell$. Let $\delta$ be the dimension of the span of $L(f_1), \dots, L(f_d)$. If $\delta=d$ we are done, so suppose that $L(f_1), \dots, L(f_m)$ are independent, but that $L(f_{m+1})$ is a linear combination of $L(f_1), \dots, L(f_m)$ for some $m < \delta$. Then for some $c > 1$, $L(f_{m+c})$ is not a linear combination of $L(f_1), \dots, L(f_{m+c-1})$.

Now deform $p$ along $\ell$ to first order and change coordinates so that $p$ is still given by $[1,0, \dots, 0]$ and $\ell$ is $V(x_2, \dots, x_n)$. This corresponds to the change of coordinates $x_1 \mapsto x_1 + \epsilon x_0$, which changes $L(f_i)$ to $L(f_i)+ \epsilon i L(f_{i+1})$, where $L(f_{d+1})$ is defined to be $0$. Under this deformation, $L(f_{m+c-1})$ becomes $L(f_{m+c-1})+ \epsilon (m+c-1) L(f_{m+c})$, which is not a linear combination of $L(f_1), \dots, L(f_m)$. This contradicts generality of $p$. Thus, the result follows.
\end{proof}

We observe that the integer $\delta$ from Lemma \ref{lem-linesDependence} is the codimension $T_{\ell} F^p(X)$ in $T_{\ell} F^p(\PP^n)$, which by deformation theory is the same as $H^0(N_{\ell/X}(-p))$, i.e. $H^0(N_{\ell/X}(-1))$. Thus, the integer $\delta$ can be computed via the equation $h^0(N_{\ell/X}(-1)) = n-1-\delta$.

\begin{theorem}
\label{thm-nonfreelines}
Let $X$ be a smooth degree $d$ hypersurface in $\PP^n$. Suppose $S \subset F_1(X)$ is a family of lines sweeping out an irreducible subvariety $Y \subset X$. Let $\ell \in S$ be a general line of $S$ and $p \in \ell$ be a general point. Let $\delta = n-1-h^0(N_{\ell/X}(-1))$. Then at least one of the following is satisfied: 
\begin{enumerate}
\item $\delta =d$ and $N_{\ell/X}$ is globally generated or
\item $\dim (S \cap F^p(X)) \leq \delta -1$ and $\dim S \leq \dim Y + \delta-2$.
\end{enumerate}
\end{theorem}
\begin{proof}
Since $N_{\ell/X}(-1)$ has rank $n-2$ and degree $-d+1$, we see that $H^0(N_{\ell/X}(-1))$ has dimension at least $n-1-d$. Since $N_{\ell/X}$ injects into $\OO(1)^{n-1}$, every summand of $N_{\ell/X}(-1)$ has degree at most 0, and $h^0(N_{\ell/X}(-1)) = n-1-d$ if and only if $N_{\ell/X}$ is globally generated. Thus, if $\delta$ equals $d$, we see that $N_{\ell/X}$ is globally generated. 

Now suppose $\delta$ is less than $d$. Select a general homogeneous coordinate $x_0$. Since $x_0$ is general, $V(f_d) = V(f,x_0)$ is smooth. Let $\uU$ be the incidence-correspondence of pairs $(p,\ell)$ where $p$ is in $\ell$ and $\ell$ is in $S$, and let $\ev: \uU \to Y$ be projection onto the first coordinate. Since $p$ and $\ell$ are general, we may bound the dimension of $\uU$ by computing the dimension of the component of $F^p(X)$ containing $\ell$. By Lemma \ref{lem-linesDependence}, $T_{\ell} V(f_d)$ contains $T_{\ell} V(f_1, \dots, f_\delta)$, and since $\ell$ is general, it follows $V(f_d)$ is tangent to $V(f_1, \dots, f_\delta)$ at every point of $\ev^{-1}(p)$. By Lemma \ref{lem-polys}, this implies that the dimension of the component of $F^p(X)$ containing $\ell$ is at most $\delta-1$ as required. Thus, $\dim \uU \leq \dim Y + \delta -1$. Since $\dim \uU = \dim S +1$, the result follows.
\end{proof}

From this, we deduce the de Jong-Debarre Conjecture for $n \geq 2d-4$.

\begin{corollary}
If $X \subset \PP^n$ is a smooth hypersurface of degree $d$ with $n \geq 2d-4$, then every component of $F_1(X)$ has the expected dimension $2n-d-3$.
\end{corollary}
\begin{proof}
First observe that the result is well-known for $d < 4$, so without loss of generality, we may assume $d \geq 4$.

Let $S \subset F_1(X)$ be a component. If $S$ contained free lines, it would have the expected dimension, so suppose $S$ consists of non-free lines, sweeping out a subvariety $Y \subset X$. By Theorem 3.2b of \cite{beheshtid8}, we see that $\dim Y \leq n-3$. Let $\ell \in S$ be general and let $h^0(N_{\ell/X}(-p)) = n-1-\delta$. If $\delta=d-1$, then $N_{\ell/X} = \OO(-1) \oplus \OO^{\oplus d-3} \oplus \OO(1)^{\oplus n-d}$. In particular, it has no $H^1$, implying that $S$ has the expected dimension. Thus, we need only consider the case where $\delta \leq d-2$. By Theorem \ref{thm-nonfreelines}, we see that $\dim S \leq \dim Y + \delta -2 \leq n-3 + d-2-2 = n+d-7$. This will be at most $2n-d-3$ if $n+d-7 \leq 2n-d-3$ or $n \geq 2d-4$.
\end{proof}

\subsection{Irreducibility}

We now develop some of the necessary techniques needed to prove that $F_1(X)$, or more generally, $F_k(X)$, is irreducible. We start with a folklore result that we prove for lack of a reference. We say that a scheme $Z$ is \emph{connected in dimension $r$} if $Z \setminus W$ is connected for any scheme $W$ of dimension less than $r$. Being connected in dimension $0$ is the same as being connected. We need the following from \cite{SGA2}, Exp. XIII, (2.1) and (2.3).

\begin{proposition}
\label{cor-connCIs}
Let $Z = V(h_1, \dots, h_c)$ be the vanishing locus of homogeneous polynomials $f_1, \dots, f_c$ on $\PP^n$. Then $Z$ is connected in dimension $n-c-1$.
\end{proposition}

We can apply Proposition \ref{cor-connCIs} to $F^p(X)$ when it has dimension at least 1, i.e., when $n \geq d+2$.

\begin{corollary}
\label{cor-irrLineFibers}
For $n \geq d+2$, $X \subset \PP^n$ smooth of degree $d$ and $p \in X$ general, $F^p(X)$ is smooth and irreducible.
\end{corollary}
\begin{proof}
The fact that $F^p(X)$ is smooth follows from Proposition \ref{prop-globGenGeneralPoint} and the relation of $T_{\ell} F^p(X) = H^0(N_{\ell/X}(-p))$. Thus, it remains to show that $F^p(X)$ is connected. This follows immediately from the description of $F^p(X)$ as $V(f_1, \dots, f_d)$ and Proposition \ref{cor-connCIs}.
\end{proof}

\begin{corollary}
\label{cor-irrLines}
If $n \geq 2d-1$, $n \geq 4$ and $X$ is a smooth degree $d$ hypersurface in $\PP^n$, then $F_1(X)$ is irreducible of the expected dimension.
\end{corollary}
\begin{proof}
Consider the space $\uU$ of pairs $(p,\ell)$ of points $p$ lying in lines $\ell$ that are contained in $X$. Since $\uU$ is a $\PP^1$ bundle over $F_1(X)$, the irreducible components of $\uU$ will be in bijection with the irreducible components of $F_1(X)$. Let $\ev: \uU \to X$ be the natural evaluation map. Note that $2d-1 \geq d+2$ given our conditions on $n$ and $d$, so by Corollary \ref{cor-irrLineFibers}, $F^p(X)$ is smooth and irreducible for a general $p \in X$. Thus, by Proposition \ref{prop-globGenGeneralPoint} it will be enough to show that any component of $\uU$ dominates $X$. 

Every component of $\uU$ has dimension at least $2n-d-2$ since every component of $F_1(X)$ has dimension at least $2n-d-3$. Let $\uU_S \subset \uU$ be the locus where the fibers of $\ev$ have dimension larger than $n-d-1$, the relative dimension of the map. We wish to show that $\uU_S$ does not contain a component of $\uU$. To get a contradiction, suppose that it does. Let $S$ be the image of $\uU_S$ in $F_1(X)$. Let $\ell \in S$ be general, let $p \in \ell$ be general and let $\delta = n-1-h^0(N_{\ell/X}(-1))$. By assumption, $\delta \leq d-1$. By Theorem \ref{thm-nonfreelines}, we see that $F^p(X) \cap S$ has dimension at most $\delta -1 \leq d-2$. If $n \geq 2d-1$, then $d-2 \leq n-d-1$, which contradicts our choice of $\uU_S$. The result follows.
\end{proof}

\section{Expected dimension of $k$-planes}
\label{sec-kplanes}
\subsection{Examples}
We start by providing some examples of hypersurfaces with larger-than-expected families of $k$-planes. These help show that our results are within a factor of $k+1$ of the optimal bound.

The first example is hypersurfaces containing a large linear space. It is well-known (see \cite{3264}, Corollary 6.26) that smooth hypersurfaces cannot contain a linear space of more than half of their dimension, but that there are smooth hypersurfaces containing linear spaces of up to half of their dimension. Let $X \subset \PP^{2m+1}$ be a smooth, degree $d$ hypersurface containing a linear space $\Lambda$ of dimension $m$. Then $X$ will certainly contain all the $k$-planes in $\Lambda$, i.e., 
\[ \dim F_k(X) \geq (k+1)(m-k) .\]
This will be larger than the expected dimension $(k+1)(2m+1-k) - \binom{d+k}{k}$ when
\[ (k+1)(m-k) > (k+1)(2m+1-k) - \binom{d+k}{k} \]
or
\[ m+1 < \frac{1}{k+1} \binom{d+k}{k} . \]

Since $n = 2m+1$ implies $\frac{n+1}{2} = m+1$, we obtain the following corollary.

\begin{proposition}
\label{prop-counterexample}
If $n < \frac{2}{k+1} \binom{d+k}{k}-1$, then there exists a smooth hypersurface $X$ of degree $d$ in $\PP^n$ with $\dim F_k(X) > (k+1)(m-k) - \binom{d+k}{k}$. Thus, in this range the dimension of $F_k(X)$ must depend on $X$.
\end{proposition}

The next family of examples is hypersurfaces with a conical hyperplane section. Consider a smooth hypersurface $X = V(f)$ satisfying $f = g + x_0 h$, where $g$ and $h$ are polynomials, and $g$ depends only on $x_2, \dots, x_n$. Then the intersection of $X$ and $V(x_0)$ will be a cone, with vertex $[0,1,0, \dots, 0]$, so the space of lines in $X \cap V(x_0)$ will have dimension at least $n-3$. If $d > n$, this will be larger than $2n-d-3$.

\begin{proposition}
\label{prop-dJBoptimal}
For every $d > n$, there exists a smooth hypersurface $X$ of degree $d$ in $\PP^n$ with $\dim F_1(X) > 2n-d-3$. Thus, Conjecture \ref{conj-deJongDebarre} is optimal.
\end{proposition}

Furthermore, a hypersurface $X$ with an $n-3$-dimensional family of lines through a given point $p$ will have a larger than expected family of genus zero stable maps: 
if $e \geq 3$, $C$ is a curve obtained by attaching $e$ irreducible rational curves $C_1, \dots , C_e$ to an irreducible rational curve $C_0$ at distinct points, and $f: C\to X$ is the map which contracts $C_0$ to $p$ and sends each $C_i$, $1 \leq i \leq e$, isomorphically onto a line through $p$, then $(C,f)$ is a stable map 
and the family of such maps is of dimension $e(n-3)+e-3$ which is larger than the expected dimension $e(n-d+1)+n-4$ when
\[ d > 3+ \frac{n-1}{e}. \]
This shows the bound $\frac{e+n}{e+1}= 1 + \frac{n-1}{e+1}$ from Theorem \ref{higherdeg} is almost sharp.

\subsection{Local Equations for the Fano scheme}
We wish to explicitly write down equations for the space of $k$-planes in $X$ containing a single $k-1$ plane $\Lambda$. 

We start by reviewing some notation from $\PP^n$. Let $\Lambda$ be a $k-1$-plane in $\PP^n$. The space of $k$-planes in $\PP^n$ containing $\Lambda$ can naturally be identified with $\PP^{n-k}$ by writing any $k$-plane containing $\Lambda$ parametrically as $\Phi = [x_0, \dots, x_{k-1}, a_k t, a_{k+1} t, \dots, a_n t]$. Here, $[x_0, \dots, x_{k-1}, t]$ are coordinates on $\PP^k$ and $[a_k, \dots, a_n]$ are the coordinates on the $\PP^{n-k}$ of $k$-planes containing $\Lambda$.

Now suppose $\Lambda$ is contained in a degree $d$ hypersurface $X$. We consider the equations on $\PP^{n-k}$ cutting out the space of $k$-planes containing $\Lambda$ that lie in $X$. Let $T$ be the set of all multisets $I$ on the numbers $0$ through $k-1$ such that $|I| \leq d-1$. Observe that $|T| = \binom{d+k-1}{k}$. Given a multiset $I$ in $T$, we have a unique monomial $x^I$. For instance $x^{\{1,1,2,3 \}} = x_1^2 x_2 x_3$. Since $\Lambda \subset X$, we know that $f$ will be a sum of monomials each of which is divisible by at least one of $x_k, \dots, x_n$, so we can write 
\begin{equation}
\label{eq-cIs}
f = \sum_{I \in T} c_I x^I,
\end{equation} 
where each $c_I$ is a homogeneous polynomial in $x_k, \dots, x_n$ of degree $d-|I| \geq 1$. Let $F^{\Lambda}(X)$ be the space of $k$-planes lying in $X$ that contain $\Lambda$.

\begin{proposition}
The $c_I$ are the equations that cut out $F^{\Lambda}(X)$ in the $\PP^{n-k}$ of $k$-planes containing $\Lambda$.
\end{proposition}

\begin{proof}
We can write any $k$-plane containing $\Lambda$ as $\Phi = [x_0, \dots, x_{k-1}, a_k t, a_{k+1} t, \dots, a_n t]$, where $[x_0, \dots, x_{k-1}, t]$ are coordinates on $\PP^k$ and $[a_k, \dots, a_n]$ are the coordinates on the $\PP^{n-k}$ of $k$-planes containing $\Lambda$. Plugging these into $f$, we get
\[ \sum_{I \in T} c_I(a_k, \dots, a_n) x^I t^{d-|I|} . \]
Thus, $f$ will vanish on $\Phi$ precisely when all of the $c_I$ vanish.
\end{proof}

For future reference it will be useful to talk about the tangent space to $F^{\Lambda}(X)$ at a point $\Phi$. Given a polynomial $c_I$ in $x_k, \dots, x_n$, a point $\Phi \in F^{\Lambda}(X)$, and a homogeneous coordinate $x_k$ on $\PP^{n-k}$, we can expand $c_I$ as a power series around $\Phi$. let $L(c_I)$ be the linear part of $c_I$ near $\Phi$, which will be independent of the chosen homogeneous coordinate $x_k$. Then the tangent space to  $F^{\Lambda}(X)$ at $\Phi$ will be cut out by $\{ L(c_I) \}_{I \in T}$. For a general hypersurface containing $\Lambda$, the $c_I$ will all impose independent conditions, so $\dim F^{\Lambda}(X)$ will be $n-k-\binom{d+k-1}{k}$ in this case. We refer to $n-k-\binom{d+k-1}{k}$ as the \emph{expected dimension} of $F^{\Lambda}(X)$.

\begin{lemma}
\label{lem-cEmptySmooth}
Let $X$ be a hypersurface with singular locus of dimension $s$ containing a $(k-1)$-plane $\Lambda$. For a general choice of coordinates $x_0, \dots, x_{k-1}$ on $\PP^n$, the singular locus of $V(c_{\emptyset}) \subset \PP^{n-k}$ will have dimension at most $s' = \max \{s-k,-1\}$. In particular, if $X$ is smooth then $V(c_{\emptyset})$ will be as well.
\end{lemma}
\begin{proof}
From the equation of $f$, we see that $V(c_{\emptyset}) = X \cap V(x_0, \dots, x_{k-1})$. Since $x_0, \dots, x_{k-1}$ are general in $\PP^n$, $X \cap V(x_0, \dots, x_{k-1})$ will have singular locus of dimension $\max \{s-k,-1 \}$. 
\end{proof}

\subsection{Dimension}

Now we prove that high-dimensional smooth hypersurfaces contain only expected dimensional families of $k$-planes. We proceed with the analog of Lemma \ref{lem-linesDependence} in the case of $k$-planes. In order to state it, we need the following definition. A subset $T' \subset T$ is a \emph{downward set} if whenever $I \in T'$ with $|I| < d-1$, $I \cup \{j \} \in T'$ for all $j \in \{0, \dots, k-1 \}$. Note that the only downward set containing $\emptyset$ is $T$ itself. 

\begin{lemma}
\label{lem-downwardSet}
Let $X = V(f)$ be a hypersurface, and $\Phi \subset X$ a $k$-plane. Then there is a downward set $T' \subset T$ such that for a general $(k-1)$-plane $\Lambda \subset \Phi$, $\{L(c_I) | \: I \in T' \}$ form a basis for $\spa \{ L(c_I) | \: I \in T \}$. In particular, one of the following must hold:
\begin{enumerate}
\item $V(c_{\emptyset})$ is tangent to $V(\{ c_I | \: I \textrm{ is not empty} \})$ at the point corresponding to $\Phi$.
\item The elements of $\{ c_I \}_{I \in T}$ all meet transversely at $\Phi$.
\end{enumerate}

\end{lemma}

\begin{proof}
Fix $\Phi = V(x_{k+1}, \dots, x_n)$. Choose $\Lambda \subset \Phi$ general and choose coordinates so that $\Lambda = V(x_k, \dots, x_n)$. As we deform $\Lambda$ to $V(x_k-\epsilon \sum_{i=0}^{k-1} a_i x_i, x_{k+1}, \dots, x_n)$, we can preserve the choice of coordinates by taking $x_k \mapsto x_k + \epsilon \sum_{i=0}^{k-1} a_i x_i$. Let $L(c_I)$ be the linear part of the expansion of $c_I$ around $\Phi$. In other words, $L(c_I)$ will be the coefficient of $x^I x_k^{d-1-|I| }$ in the expression for $f$.

We claim that there is a downward subset $T' \subset T$ such that $\{ L(c_I) | \: I \in T' \}$  form a basis for $\spa \{ L(c_I) | \: I \in T \}$. If $L(c_I) = 0$ for all $I$, then we can take $T' = \emptyset$, and note as an aside that $f$ will vanish to order at least $2$ along $\Phi$ in this case.

Let $T_1 \subset T$ be such that $\{ L(c_I)| \: I \in T_1 \}$ is a basis for $\spa \{ c_I | \: I \in T \}$. Let $T_2 \subset T_1$ be a largest downward subset. If $T_1$ is not a downward set, then there must be some $J \in T_1$ and some $m \in \{0, \dots, k-1 \}$ such that $L(c_J)$ is independent of $\{ L(c_I) | \: I \in T_2 \}$ but $L(c_{J \cup \{m\}})$ is dependent on $\{ L(c_I) | \: I \in T_1 \}$. We can choose $J$ so that $|J|$ is as large as possible. Now deform $\Lambda$ using the change of coordinates $x_k \to x_k + \epsilon x_m$. Under this change of coordinates, we have $L(c_I) \mapsto L(c_I) + \epsilon (d-1-|I|) L(c_{I \setminus \{m\}})$ if $m \in I$ and $L(c_I) \mapsto L(c_I)$ otherwise. Under this deformation, we see that $L(c_{J \cup \{m \}})$ will become independent of $\{ L(c_I) | \: I \in T_2 \}$, contradicting generality of $\Lambda$, since $|J|$ was as large as possible. Thus, $T_1$ must be a downward set.
\end{proof}

\begin{remark}
\label{rem-PhiLambdaDim}
Note that $\dim T_{\Phi} F^{\Lambda}(X)$ does not depend on the choice of $\Lambda \subset \Phi$, since $H^0(N_{\Phi /X}(-\Lambda)) = H^0(N_{\Phi /X}(-1))$.
\end{remark}

\begin{theorem}
\label{thm-kPlanesContLambda}
Let $\Phi \subset X$ be a $k$-plane contained in a hypersurface $X$ with singular locus of dimension $s$. Then for a general $(k-1)$-plane $\Lambda$ in $\Phi$, one of the following two conditions hold:
\begin{enumerate}
\item \label{case-smooth} $\Phi$ is a smooth point of $F^{\Lambda}(X)$, and the component of $F^{\Lambda}(X)$ containing $\Phi$ has the expected dimension.
\item \label{case-singular} For some $\delta < \binom{d+k-1}{k}$, $\Phi$ is an element of the set $S \subset F^{\Lambda}(X)$ consisting of $\Theta$ such that $\dim T_{\Theta} F^{\Lambda}(X) = \dim T_{\Phi} F^{\Lambda}(X) = n-k-\delta$ and any component of $S$ containing $\Phi$ has dimension at most $\delta+\max\{s-k, -1\}$.
\end{enumerate}
\end{theorem}
\begin{proof}
If $T_{\Phi} F^{\Lambda}(X)$ has the expected dimension $n-k-\binom{d+k-1}{k}$, then we are in case \ref{case-smooth} and are done, so suppose not. Let $\delta$ be defined by $\dim T_{\Phi} F^{\Lambda}(X) = n-k-\delta$, and let $S_0$ be a component of $S = \{ \Theta \in F^{\Lambda} | \: \dim T_{\Theta} F^{\Lambda} (X) = n-k-\delta \}$ containing $\Phi$. By Lemma \ref{lem-cEmptySmooth}, $V(c_{\emptyset})$ will have singular locus of dimension at most $\max \{ s-k, -1\}$. By Lemma \ref{lem-downwardSet}, there is a downward set $T' \subset T$ with $|T'| = \delta$ such that $V(c_{\emptyset})$ is tangent to $V(\{ c_I | \: I \in T' \})$ at a general point of $S_0$. By Lemma \ref{lem-polys} this can only happen on a locus of dimension at most $\delta + \max\{ s-k, -1\}$. The result follows.
\end{proof}

For the corollary, we need to set up some notation. Let $B$ be a subvariety of $F_k(X)$, and let $\uU_B$ be defined by
\[ \uU_B = \{(\Lambda, \Phi) | \: \Lambda \in \gG(k-1,n), \Phi \in B, \Lambda \subset \Phi \} ,\]
with $\pi_1$ and $\pi_2$ the two projections.

\begin{corollary}
\label{cor-dimBadLocus}
Suppose $X \subset \PP^n$ is a degree $d$ hypersurface with singular locus of dimension at most $s$. Let $B \subset F_k(X)$ be a collection of $k$-planes $\Phi$ with $\dim T_{\Phi} F^{\Lambda}(X) = n-k - \delta$ for any $(k-1)$-plane in $\Phi$ and some $\delta < \binom{d+k-1}{k}$. Then we have
\[  \dim B \leq \dim \pi_1(\uU_B) + \delta + \max\{s-2k,-k-1\} .\]
\end{corollary}
\begin{proof}
By considering the projection $\pi_2$, we see that the dimension of $\uU_B$ is $\dim B + k$. By Theorem \ref{thm-kPlanesContLambda}, the fibers of $\pi_1$ over a general point of $\pi_1(\uU_B)$ will have dimension at most $\delta+\max\{s-k,-1\}$. Thus, 
\[ \dim B = \dim \uU_B - k \leq \dim \pi_1(\uU_B) + \delta+\max\{s-2k, - k-1\} .\]
\end{proof}

\begin{corollary}
\label{cor-expDimkplanes}
If $X \subset \PP^n$ is a degree $d$ hypersurface with singular locus having dimension at most $s$ and $n \geq 2\binom{d+k-1}{k}+\max\{s-1,k-2\}$, then $F_k(X)$ has the expected dimension $(k+1)(n-k) - \binom{k+d}{d}$.
\end{corollary}
\begin{proof}
We prove the result by induction on $k$. For $k=0$, the result is clear. Now suppose the result is known for $F_{k-1}(X)$. Let $B$ be an irreducible component of $F_{k}(X)$, and let 
\[ \uU_B = \{(\Lambda, \Phi) | \: \Lambda \in \gG(k-1,n), \Phi \in B, \Lambda \subset \Phi \} ,\]
with $\pi_1$ and $\pi_2$ the two projections. If the general fiber of $\pi_1$ has the expected dimension $n-k-\binom{d+k-1}{k}$, then 
\[ \dim B = \dim \uU_B - k = \dim F_{k-1}(X) + n-k-\binom{d+k-1}{k} - k \]
\[ = k(n-k+1) - \binom{d+k-1}{k-1} + n-k - \binom{d+k-1}{k} - k = (k+1)(n-k) - \binom{d+k}{k} . \]

Thus, it is enough to show that a general fiber of $\pi_1$ has the expected dimension. Let $(\Lambda, \Phi)$ be a general point of $\uU_B$. Let $\delta$ be $n-k-\dim T_{\Phi} F^{\Lambda}(X)$. If $\delta = \binom{d+k-1}{k}$ the result follows, so assume $\delta \leq \binom{d+k-1}{k}-1$. By Theorem \ref{thm-kPlanesContLambda}, the fibers of $\pi_1$ over a general point of $\pi_1(\uU_B)$ have dimension at most $\delta + \max\{s-k,-1\} \leq \binom{d+k-1}{k}+\max\{s-k,-1\}-1$. This will be at most $n-k-\binom{d+k-1}{k}$ if
\[ \binom{d+k-1}{k}+\max\{s-k,-1\}-1 \leq n-k-\binom{d+k-1}{k}  \]
or if
\[ n \geq 2 \binom{d+k-1}{k} + \max\{s-1, k-2 \} .\]
\end{proof}

\subsection{Irreducibility}
We need the following corollary of Proposition \ref{cor-connCIs}.

\begin{corollary}
For $n \geq k+\binom{d+k-1}{k}+1$, the variety $F^{\Lambda}(X)$ is connected in dimension $n-k-\binom{d+k-1}{k}-1$.
\end{corollary}
\begin{proof}
This follows immediately from Corollary \ref{cor-connCIs} and the description of $F^{\Lambda}(X)$ as $V(\{c_I\})$.
\end{proof}

From this we can deduce irreducibility for the space of $k$-planes.

\begin{theorem}
\label{thm-irrkplanes}
If $n \geq 2 \binom{d+k-1}{k} + \max\{s+1,k \}$, and $X \subset \PP^n$ is a degree $d$ hypersurface with singular locus of dimension $s$, then $F_k(X)$ is irreducible of the expected dimension.
\end{theorem}
\begin{proof}
By Corollary \ref{cor-expDimkplanes}, we know that $F_k(X)$ has the expected dimension, so it remains to show irreducibility. Let $\uU$ be the set of pairs $(\Lambda, \Phi)$ such that $\Lambda$ is a $k$-plane in $\Phi \in F_k(X)$. There are two parts to the proof. First, we show that any irreducible component of $\uU$ must dominate $F_{k-1}(X)$. Second, we show that a general fiber of $\pi_1: \uU \to F_{k-1}(X)$ is irreducible. The result will follow by induction.

For any set $S \subset F_k(X)$, let $\uU_S$ be the space of pairs $(\Lambda, \Phi)$ with $\Phi$ a $k$-plane in $S$ and $\Lambda$ a $(k-1)$-plane in $\Phi$.  Let $B$ be an irreducible component of $F_k(X)$ such that $\uU_B$ does not dominate $F_{k-1}(X)$. Then since $\uU_B$ must have dimension at least the expected dimension but $\pi_1$ is not dominant, it follows that the fibers of the map $\pi_1|_{\uU_B}$ must have dimension larger than $n-k-\binom{d+k-1}{k}$.

Let $B' \subset F_k(X)$ be the space of $\Phi$ such that for any $(k-1)$-plane $\Lambda$ in $\Phi$, $\dim T_{\Phi} F^{\Lambda}(X) > n-k-\binom{d+k-1}{k}$. By Remark \ref{rem-PhiLambdaDim}, if this condition on the dimension of $T_{\Phi} F^{\Lambda}(X)$ is true for a single pair $(\Lambda, \Phi)$, it will be true for any $\Lambda' \subset \Phi$. From the discussion above, it follows that $B$ must be an irreducible component of $B'$. Let $\delta = n-k- \dim T_{\Phi} F^{\Lambda}(X)$ for a general pair $(\Lambda, \Phi) \in \uU_B$. By construction of $B'$, $\delta$ is necessarily greater than $\binom{d+k-1}{k}$. By Corollary \ref{cor-dimBadLocus},
\[ \dim B \leq \dim \pi_1(\uU_B) + \delta+\max\{s-2k,-k-1 \} \leq k(n-k) - \binom{d+k-1}{k-1} + \delta + \max\{s-k, -1\} ,\]
where the last inequality follows from the fact $\pi_1(\uU_B) \leq F_{k-1}(X) = k(n-k+1) - \binom{d+k-1}{k-1}$ by Corollary \ref{cor-expDimkplanes}.

This means $B$ cannot be an irreducible component of $F_k(X)$ if
\[ k(n-k) - \binom{d+k-1}{k-1} + \delta + \max\{s-k, -1\} \leq (k+1)(n-k) - \binom{d+k}{k} -1  \]
or equivalently, since $\delta \leq \binom{d+k-1}{k}-1$,
\[ n \geq \binom{d+k}{k} - \binom{d+k-1}{k-1} + \binom{d+k-1}{k} + k + \max\{s-k,-1\} = 2\binom{d+k-1}{k} + \max\{s, k-1 \} . \]
This holds by our assumptions on $n$.

Thus, it suffices to show that a general fiber of $\alpha: \uU_B \to F_{k-1}(X)$ is irreducible. We know that such a fiber is connected in dimension $n-k-\binom{d+k-1}{k}-1$. The singular locus of a general fiber $\alpha^{-1}(\Lambda)$ has dimension at most 
\[ \dim \uU_B - \dim F_{k-1}(X) \leq \delta + \max\{s-k,-1\} \leq \binom{d+k-1}{k}-k-2 + \max\{s+1,k \} . \]
Since $n \geq 2\binom{d+k-1}{k} + \max\{s+1,k \}$, we see that the singular locus of $\alpha^{-1}(\Lambda)$ has dimension at most
\[ n-k-\binom{d+k-1}{k} - 2 .\]
Since $\alpha^{-1}(\Lambda)$ is connected in dimension $n-k-\binom{d+k-1}{k}-1$, this shows that $\alpha^{-1}(\Lambda)$ must be irreducible.
\end{proof}

\section{Unirationality}
\label{sec-unirationality}
In this section, we consider the unirationality of hypersurfaces. We find explicit, closed-form bounds for when arbitrary smooth hypersurfaces are unirational, using the technique of \cite{HMP} based on a construction described in \cite{paranjape-srinivas}. Our improved bounds come from the new result on $k$-planes, but for the reader's convenience, we briefly describe the construction from \cite{HMP}. The referee pointed out to us that a modification to the argument in \cite{paranjape-srinivas} combined with results of \cite{starr-kplanes} provides a slightly better bound than the one we obtain in this paper. We start with a Bertini Lemma from \cite{HMP}.

\begin{lemma}[cf Lemma 4.1 from \cite{HMP}]
\label{lem-Bertini}
Consider a linear series $D = \{D_p \subset \PP^n \}_{p \in \PP^m}$ of hypersurfaces in $\PP^n$. Let $b$ be the dimension of the base locus of $D$ (where $b=-1$ if the base locus is empty). Set
\[ S_k = \{ p \in \PP^m | \: (\dim D_p)_{\sing} \geq k + b \} . \]
Then $\dim S_k \leq m-k$.
\end{lemma}

\subsection{Residual hyperplanes}
Recall the basic setup from \cite{HMP}. In order to make the unirationality proof from \cite{paranjape-srinivas} work, we need to be able to work in families of hypersurfaces in varying families of projective spaces. Let $B$ be a scheme. Then one can construct a family of projective spaces over $B$ by taking a vector bundle $E$ on $B$ and taking the projectivization $\PP(E)$ of it. A \emph{family of degree $d$ hypersurfaces} in those projective spaces is the zero locus of a section $\sigma$ of $\OO_{\PP(E)}(d)$ such that $\sigma$ does not vanish on any fibers of the projection $\PP(E) \to B$. A \emph{family of $k$-planes} in $\PP(E)$ is simply $\PP(F)$ where $F \subset E$ is a rank $k+1$ sub-vector bundle.

Recall the geometry of taking a residual hyperplane section. Let $X = V(f)$ be a hypersurface containing a linear space $\Gamma = V(x_{k+1}, \dots, x_n)$. Given a $k+1$-plane $\Phi$ containing $\Gamma$ but not contained in $X$, we can intersect $\Phi$ with $X$ to obtain a degree $d$ hypersurface in $\Phi$. This hypersurface will be $\Gamma \cup Y_{\Phi}$ for some degree $d-1$ hypersurface $Y_{\Phi}$. We call $Y_{\Phi}$ the \emph{residual hypersurface} to $\Gamma$. For later convenience, we describe how to write $Y_{\Phi}$ in coordinates. Let $V(f)$ be a hypersurface containing a linear space $\Gamma = V(x_{k+1}, \dots, x_n)$. Recall from (\ref{eq-cIs}) that we can expand the equation of $f$ around $\Gamma$, getting $f = \sum_I c_I x^I$. Plugging in a $k+1$-plane $ \Phi = [x_0, \dots, x_k, t a_{k+1}, \dots, a_n t]$ containing $\Gamma$ to the equation of $f$, we get
\[ f(\Phi) = \sum_{I \in T} c_I(a_{k+1}, \dots, a_n) x^I t^{d-|I|} . \]
Since $d-|I| \geq 1$ for all $I \in T$, we can divide $f(\Phi)$ by $t$ to get the equation of the residual hypersurface to $\Gamma$. Since $\Gamma$ is cut out in $\Phi$ by the equation $t=0$, we see that if we intersect $Y_{\Phi}$ with $\Gamma$, all terms with $|I| < d-1$ vanish. Thus, the equation of $Y_{\Phi} \cap \Gamma$ will be 
\[ \sum_{I \in T, |I| = d-1} c_I(a_{k+1}, \dots, a_n) x^I . \]
For $|I| = d-1$, $c_I$ will be linear, so we see that the equations of the intersections $\Gamma \cap Y_{\Phi}$ vary linearly with the coordinates $a_{k+1}, \dots, a_n$. This means that the $\Gamma \cap Y_{\Phi}$ form a linear series on $\Gamma$.

\begin{corollary}
\label{cor-linearSeries}
For $X$ a smooth hypersurface containing a $k$-plane $\Gamma$, construct for each $k+1$ plane $\Phi$ not lying in $X$ but containing $\Gamma$ the hypersurfaces $Y_{\Phi}$ as above. Then the hypersurfaces $Y_{\Phi} \cap \Gamma$ form a basepoint free linear series on $\Gamma$.
\end{corollary}
\begin{proof}
The fact that these form a linear series follows from the above discussion. The only remaining thing to show is that this linear series is basepoint free. The points of $Z_{\Phi} = \Gamma \cap Y_{\Phi}$ will be singular points of $Y_{\Phi} \cup \Gamma = X \cap \Phi$. Thus, a basepoint of the linear series will be a singular point of $X \cap \Phi$ for all $\Phi$. This is impossible for smooth $X$.
\end{proof}

The basic inductive step in the argument is the following.

\begin{theorem}
\label{thm-unirationality}
Let $(\PP(E), \Gamma_b, X_b)$ be a family of smooth, $n-1$-dimensional, degree $d$, $k$-planed hypersurfaces over $B$, with $\xX$ the total space of the family. Then if $k \geq 1+2\binom{d+r-2}{d-2}+\binom{d+r-1}{r}$, there is a family of smooth, degree $d-1$, $r$-planed hypersurfaces $(\PP(E'), \Lambda_b, Y_b)$ over a base $B'$ having total space $\yY$, together with a surjective map $\beta: B' \to B$ such that the following diagram commutes
\catcode`\@=10
\newdimen\cdsep
\cdsep=3em

\def\cdstrut{\vrule height .25\cdsep width 0pt depth .12\cdsep}
\def\@cdstrut{{\advance\cdsep by 2em\cdstrut}}

\def\arrow#1#2{
  \ifx d#1
    \llap{$\scriptstyle#2$}\left\downarrow\cdstrut\right.\@cdstrut\fi
  \ifx u#1
    \llap{$\scriptstyle#2$}\left\uparrow\cdstrut\right.\@cdstrut\fi
  \ifx r#1
    \mathop{\hbox to \cdsep{\rightarrowfill}}\limits^{#2}\fi
  \ifx l#1
    \mathop{\hbox to \cdsep{\leftarrowfill}}\limits^{#2}\fi
}
\catcode`\@=10

\cdsep=3em
$$
\begin{matrix}
 \yY & \arrow{r}{\alpha} & \xX \cr
 \arrow{d}{} & & \arrow{d}{\pi} \cr
 B' & \arrow{r}{\beta} & B \cr
\end{matrix}
$$
and $\alpha: \yY \to \xX$ dominates each fiber $X_b$ of $\pi$. Moreover, if $B$ is rational, then $B'$ is rational as well.
\end{theorem}

\begin{proof}
Let $\Gamma$ be given by $\PP(F)$ for a subbundle $F \subset E$. Then the space of pairs $(\Phi,b)$ where $b \in B$ and $\Phi$ is a $k+1$-plane in $\PP(E)_b$ containing $\Gamma_b$ is parameterized by $\PP(E/F)$. For each $(\Phi,b)$, we have the residual hypersurface $Y_{\Phi,b}$, and the intersection $Z_{\Phi,b} = Y_{\Phi,b} \cap \Gamma_b$. By Corollary \ref{cor-linearSeries}, the hypersurfaces $Z_{\Phi,b}$ form a basepoint free linear series on $\Gamma_b$ for each $b$.

Now consider the relative Grassmannian $\gG(r,\Gamma)$ consisting of pairs $(b,\Lambda)$ such that $\Lambda \subset \Gamma_b$ is an $r$-plane. From this, we can form the incidence-correspondence 
\[ T = \{ (\Lambda, \Phi, b) | \: \Phi \in \PP(E/F)_b, \: \Lambda \subset Z_{\Phi,b}, \: \Lambda \textrm{ is an $r$-plane} \}. \] 
We check that the fiber of $T$ over any point $b \in B$ is irreducible. We do this by considering the projection $\pi_{2,b}: T_b \to \PP(E/F)_b = \PP^{n-k-1}$. The fiber of $\pi_{2,b}$ over a point $\Phi \in \PP^{n-k-1}$ will simply be the space of $r$-planes contained in $Z_{\Phi,b}$. By Lemma \ref{lem-Bertini}, we see that a general $Z_{\Phi,b}$ will be smooth, and that the locus of $Z_{\Phi,b}$ that are singular in dimension $s$ will have codimension at least $s+1$ in $\PP^{n-k-1}$. Thus, since $k \geq 1+2\binom{d+r-2}{d-2}+\binom{d+r-1}{d-1}$, it follows by Theorem \ref{thm-irrkplanes} that the fibers of $\pi_{2,b}$ will have dimension $(r+1)(k-r)-\binom{d+r-1}{r}$ outside a set of codimension at least $\binom{d+r-1}{d-1}$, and will always have dimension at most $(r+1)(k-r)$. Thus, the incidence correspondence $T_b$ is irreducible of dimension $n-k-1+(r+1)(k-r)-\binom{d+r-1}{d-1}$. This implies that the total space $T$ is also irreducible.

We now claim the projection $\pi_{1,b}$ from $T_b$ to $\gG(r,\Gamma_b)$ is dominant. It is $\binom{d+r-1}{r}$ conditions for a degree $d-1$ hypersurface to contain an $r$-plane, so $\pi_{1,b}$ will be dominant if $n -k-1 > \binom{d+r-1}{r}$. Since each $X_b$ is smooth and contains the $k$-plane $\Gamma_b$, it follows that $n-1 \geq 2k$. From the condition $k \geq 1+2\binom{d+r-2}{d-2}+\binom{d+r-1}{d-1}$, it follows that $n-k-1 > \binom{d+r-1}{d-1}$, as required.

Let $B_0 \subset \gG(r,\Gamma)$ be the open locus over which the fibers of $\pi_1:T \to \gG(r,\Gamma)$ have the minimal dimension $n-k-1-\binom{d+r-1}{d-1}$, and let $T_0 = \pi_1^{-1}(B_0)$. We know that $B_0$ surjects onto $B$ since $\pi_{1,b}$ is dominant for all $b \in B$. Because the $Z_{\Phi,b}$ form a linear series for each $b \in B$, we have a vector bundle $V$ over $B_0$ such that $T_0 = \PP(V)$. Let $B_1$ be the open locus in $\PP(V)$ over which $Y_{\Phi,b}$ is smooth. Then $B_1$ still surjects onto $B$, since for any $b \in B$ and a general $\Phi \in \PP^{n-k-1}_b$, we have that $Z_{\Phi,b}$ will be smooth, and since $Y_{\Phi,b}$ will be smooth away from the base locus of the family, we see that $Y_{\Phi,b}$ will be as well. Thus, we can take $\yY$ to be the universal point on the family of degree $d-1$ hypersurfaces $B_1$. We can picture the above construction in the following diagram.

\def\cdstrut{\vrule height .25\cdsep width 0pt depth .12\cdsep}
\def\@cdstrut{{\advance\cdsep by 2em\cdstrut}}

\def\arrow#1#2{
  \ifx d#1
    \llap{$\scriptstyle#2$}\left\downarrow\cdstrut\right.\@cdstrut\fi
  \ifx u#1
    \llap{$\scriptstyle#2$}\left\uparrow\cdstrut\right.\@cdstrut\fi
  \ifx r#1
    \mathop{\hbox to \cdsep{\rightarrowfill}}\limits^{#2}\fi
  \ifx l#1
    \mathop{\hbox to \cdsep{\leftarrowfill}}\limits^{#2}\fi
}
\catcode`\@=10

\cdsep=3em
$$
\begin{matrix}
 \yY &=& \{ (p,b,\Phi,\Lambda) | \: p \in Y_{\Phi} \} \cr
 \arrow{d}{} & \cr
 B_1 & \subset & \PP(V) = T_0 & \subset & T = \{(b,\Phi,\Lambda) | \: \Lambda_b \subset Z_{\Phi}  \} & \arrow{r}{} & \PP(E/F) = \{(b,\Phi) \} \cr
 & & \arrow{d}{} & & \arrow{d}{} & & \arrow{d}{} \cr
 & & B_0 & \subset & \gG(r, \Gamma) = \{ \Lambda \} & \arrow{r}{} & B = \{b \}
\end{matrix}
$$

We see that $\yY$ is naturally a family of hypersurfaces in the family of projective spaces $\PP(E')$, where $E'$ is the pullback of the universal subbundle on $\PP(E/F)$, viewed as a subvariety of $\gG(k+1,E)$. Pulling back the universal subbundle from $\gG(r,\Gamma)$, we obtain a family of $r$-planes in $\yY$ as desired. Since a general point $p$ of $X_b$ will lie in a plane $\Phi$ containing an $r$-plane $\Lambda$ with $(p,b,\Phi,\Lambda) \in \yY$, we see that $\yY \to \xX$ will be fiberwise dominant.

It remains to see that $B_1$ is rational if $B$ is. We know that $B_0$ is rational, since it is an open set in a Grassmannian bundle over $B$. Thus, $B_1$ will be as well, since it is an open set in a projective bundle over a rational base.
\end{proof}

\begin{corollary}
Let $k_0(d)$ be defined by $k_0(2) = 0$ and
\[ k_0(d) = 1+2\binom{k_0(d-1)+d-2}{d-2}+\binom{k_0(d-1)+d-1}{d-1} . \]
Let 
\[ n_0(d) = \left\lceil \frac{1}{k_0(d)+1} \binom{k_0(d)+d}{d} \right \rceil + k_0(d) .\]
Then any family of smooth $k_0(d)$-planed degree $d$ hypersurfaces in $\PP(E)$ over a rational base $B$ is unirational. In particular, any smooth degree $d$ hypersurface in $\PP^n$ is unirational provided that $n \geq n_0(d)$.
\end{corollary}
\begin{proof}
We prove the result by induction on $d$. For $d=2$, the result is trivial: any smooth quadric containing a point is rational, and hence, unirational. Now suppose the result is known for $d-1$. Then by Theorem \ref{thm-unirationality}, we can construct a dominating family of $k_0(d-1)$-planed degree $d-1$ hypersurfaces that dominate the family of degree $d$ hypersurfaces. The result follows.

The second part follows from the fact that if $(k+1)(n-k) \geq \binom{k+d}{d}$, then any hypersurface of degree $d$ in $\PP^n$ contains a $k$-plane.
\end{proof}

The number $n_0(d)$ is quite large. However, this is much smaller than the bound in \cite{HMP}, which grows like a $d$-fold iteration of this exponential. We now prove upper bounds on the growth of $n_0(d)$. First we need a lemma.

\begin{lemma}
\label{lem-binomPolys}
If $d \geq 5$ and $x \geq 6$, then $\binom{x+d}{d} < \frac{1}{4} x^d$.
\end{lemma}
\begin{proof}
We have
\[ \binom{x+d}{d} = \prod_{i=1}^d \frac{x+i}{i} = x^d \prod_{i=1}^d \left(\frac{1}{x}+\frac{1}{i} \right) .\]
Since $x \geq 6$, the $i=1$ term of the product is at most $\frac{7}{6}$, the $i=2$ term is at most $\frac{2}{3}$, and the subsequent terms are at most $\frac{1}{2}$. Since $d \geq 5$, there are at least $5$ terms, which means the total product will be at most $\frac{1}{4} x^d$.
\end{proof}

This gives us the desired upper bound.

\begin{corollary}
\label{cor-unirationalityBound}
We have $k_0(2) = 0$, $k_0(3) = 4$, $k_0(4) = 66$, and $k_0(d) \leq 2^{(d-1)!}$ for $d \geq 5$. Furthermore,  $n_0(d) \leq 2^{d!}$ for all $d$.
\end{corollary}
\begin{proof}
We start by computing $k_0(2)=0$, $k_0(3) = 4$, $k_0(4) = 66$, and $k_0(5) = 1021684$. Note that $2^{4!} = 16777216$, which shows the base case $d=5$. We now prove the result by induction on $d$. Suppose $k_0(d-1) \leq 2^{(d-2)!}$. Then for $d \geq 6$
\[ k_0(d) = 1 + 2\binom{k_0(d-1)+d-2}{d-2}+\binom{k_0(d-1)+d-1}{d-1} < 4 \binom{k_0(d-1)+d-1}{d-1}.\]
By Lemma \ref{lem-binomPolys}, this last is at most
\[  k_0(d-1)^{d-1} \leq 2^{(d-2)! \cdot (d-1)} = 2^{(d-1)!} .\]
This completes the proof of the first statement. The bound for $n_0(d)$ follows from Lemma \ref{lem-binomPolys} and the definition of $n_0(d)$, with the case $d=4$ needing to be checked separately.
\end{proof}

\section{Higher degree curves}
\label{sec-higherdegree}

In this section we consider higher degree rational curves. Using our results on lines, we can prove that the space of degree $e$ rational curves on every smooth hypersurface of low degree has the expected dimension as long as $n$ is large relative to $e$. 

To compactify the space of smooth rational curves of degree $e$ on $X$, we consider the moduli space of stable maps. Let $\overline{M}_{0,r}(X,e)$ denote the Kontsevich moduli scheme parametrizing 
pointed maps $(C,p_1,\dots,p_r,f)$ where $C$ is a projective, connected, nodal curve of genus 0, $p_1, 
\dots, p_r$ are non-singular distinct points of $C$, and $f: C \to X$ is a stable morphism of total degree $e$. The 
space $\overline{M}_{0,0}(X,e)$ contains as an open subscheme the space parametrizing smooth rational curves of degree $e$ on $X$. By \cite[Theorem II.1.7]{kollar} the dimension of every irreducible component of 
$\overline{M}_{0,0}(X,e)$ is at least  $e(n+1-d)+n-4$, and if $d \leq n-1$, then $\overline{M}_{0,0}(X,e)$ has at least one irreducible component of dimension $e(n+1-d)+n-4$. We refer to the number  $e(n+1-d)+n-4$ as the {\em{expected dimension}} of $\overline{M}_{0,0}(X,e)$.
(See 
\cite{harris-roth-starr} and \cite{fulton-pandharipande} for a detailed study of these moduli spaces.)
 
In this section, we show the following:

\begin{theorem}\label{higherdeg}
Let $X$ be a smooth degree $d$ hypersurface in $\PP^n$.  Then $\overline{M}_{0,0}(X,e)$ is irreducible of the expected dimension provided that $d \leq \frac{e+n}{e+1}$. 
\end{theorem}

To prove the theorem we use a similar method as in \cite[Lemma 3.4]{dejong-starr}. Recall that a rational curve $f: \PP^1 \to X$ is \emph{free} if $f^*T_X$ is globally generated. We start with the following lemmas. 
 
\begin{lemma}\label{rat-def}  \cite[Theorem II.7.6]{kollar}. 
 Let $[(C,f)]$ in $\overline{M}_{0,0}(X,e)$ be such that the restriction of $f$ to every  component of $C$ is free. Then
\begin{itemize}
\item[(a)] $\overline{M}_{0,0}(X,e)$ has the expected dimension  at $[(C,f)]$.
\item[(b)] For any smooth point $q \in C$, there exists a deformation of $(C,f)$ smoothing the nodes of $C$ and keeping $f(q)$ fixed. 
\end{itemize}
\end{lemma}

\begin{lemma}\label{stable-bound}
If $p$ is a point of a smooth hypersurface $X$ in $\PP^n$, then the fiber of the evaluation map $ev: \overline{M}_{0,1}
(X,e) \to X$ over $p$ is at most $en-2$ dimensional. 
\end{lemma}

\begin{proof}
Assume $f: \PP^1 \to X$ is a morphism of degree $e$ with $f(q)=p$ and  $N_{f}$ is the cokernel of the map 
$T_{\PP^1} \to f^*T_X$. Then by \cite[Theorem II.1.7]{kollar}, the Zariski tangent space to the space of morphisms from $\PP^1$ to $X$ which map $q$ to $p$ is isomorphic to $H^0(f^*T_X(-q))$, and so the dimension of $ev^{-1}(p)$ at $(\PP^1, f, q)$ is at most $h^0(N_{f} (-q))$. There is an exact sequence 
$$ 0 \to f^*T_X (-q) \to f^*T_{\PP^n}(-q) \to f^* \OO_X(\deg X)(-q) \to 0.$$
By the Euler exact sequence, the image of  $H^0(f^*T_{\PP^n}(-q)) \to H^0(f^* \OO_X(d)(-q))$ contains the image of $H^0(f^*\OO_X(1)(-q)) \to H^0(f^*(\OO_X(d)(-q))$ given by any of the partial derivatives of the form defining $X$. So the image of 
$H^0(f^*T_{\PP^n}(-q)) \to H^0(f^* \OO_X(d)(-q))$ is at least $e$ dimensional. Thus
$h^0( f^*T_X (-q)) \leq h^0(f^*T_{\PP^n}(-q)) - e = en,$ and so
$h^0( N_{f} (-q)) \leq en-2$.
\end{proof}

\begin{proof}[Proof of Theorem \ref{higherdeg}]
We first prove that every irreducible component $M$ of  $\overline{M}_{0,0}(X,e)$  is of the expected dimension $e(n+1-d)+n-4$. 
To show this, we prove the following claim: the locus in $\overline{M}_{0,0}(X,e)$ parametrizing stable maps with at least one non-free component 
has dimension at most $en+d-5$. Since by our assumption $d \leq \frac{e+n}{e+1}$, we have 
$$en+d-5 <  e(n+1-d)+n-4 \leq \dim M,$$ so the claim shows that it is not possible that every curve parametrized by $M$ has a non-free component. So by part (a) of Lemma \ref{rat-def}, $M$ has the expected dimension. 

The proof of the claim is by induction on $e$. For $e=1$ this was proved in Theorem \ref{thm-nonfreelines}. Assume the statement holds for $1, \dots, e-1$, and let $N$ be an irreducible closed subscheme of  $\overline{M}_{0,0}(X,e)$ parametrizing stable maps with at least one non-free component. 
Set $r=\dim N $. If $r \leq 2n-4$, then 
$ r \leq en+d-5$ and we are done. 
So assume $r > 2n-4$, and let $Y$ be the subvariety of $X$ swept out by the images of morphisms parametrized by $N$. Since $\dim Y \leq n-1$, there are two distinct points of $Y$ with at least a $1$-dimensional family of maps parametrized by $N$ through them. So by the Bend and Break lemma \cite[Lemma 5.1]{harris-roth-starr}, there should be a map with reducible domain parametrized by $N$. Since the locus of stable maps with reducible domains in $\overline{M}_{0,0}(\PP^n,e)$ 
is of codimension 1, the locus parametrizing maps with reducible domains in $N$ is of codimension at most 1. Hence there are $e_1, e_2 < e$,  
$e_1+e_2=e$, and a subscheme $N'$ of dimension at least $r-1$ in $N$ parametrizing stable maps which can be decomposed as the union of a stable map $(C_1, f_1)$ of degree $e_1$ with at least one non-free component and a stable map $(C_2,f_2)$ of degree $e_2$. Since 
by Lemma \ref{stable-bound} the dimension of the space of degree $e_2$ stable maps through any point of $X$ is at most $e_2n-2$, 
by our induction hypothesis, we have $\dim N' \leq e_1n+d-5 + 1 + e_2n-2 = en+d-6.$ So $r \leq en+d-5$, and the claim holds for $e$. 
 
We now show by induction on $e$ that for every $e$, $\overline{M}_{0,0}(X,e)$ is irreducible. For $e=1$, the irreducibility was proved in Corollary \ref{cor-irrLines}. Assume that $e \geq 2$ and the statement holds for every degree smaller than $e$.  Let $M_1^0$ be the open subscheme of  $\overline{M}_{0,1}(X,1)$ parametrizing pointed free lines  and let $M_{e-1}^0$ be the open  subscheme of $\overline{M}_{0,1}(X,e-1)$ parametrizing pointed free stable maps with irreducible domain. Note that the evaluation maps $ev_1: M_1^0 \to X$ and $ev_{e-1}: M_{e-1}^0 \to X$ are dominant and flat by \cite[Corollary II.3.5.4]{kollar}. Moreover, general fibers of $ev_1$ are irreducible. This is because only free lines pass through a general point of $X$, so the space of lines through a general point is smooth by 
Lemma \ref{rat-def}, and since this space is a complete intersection of dimension $\geq 1$, it is also connected and is hence irreducible. 
This implies that $M_1^0 \times_X M_{e-1}^0$ is irreducible.

Let $M$  be an irreducible component of $\overline{M}_{0,0}(X,e)$. By the Bend and Break lemma, there is a stable map $f$ parametrized by 
$M$ such that the domain of $f$ is reducible. Note that by the above dimension count argument, the dimension of the locus of stable maps with reducible domain and at least one non-free component is smaller than $\dim M -1$, so there is  $f$ with reducible domain  parametrized by $M$ with the property that its restriction to every irreducible component is free, hence $[f]$ is a smooth point of $M$. We can separate the domain of $f$ into two connected components of 
degrees $e_1$ and $e_2$ for some $e_1 \leq e_2$. By part (b) of Lemma \ref{rat-def},  the restriction of $f$ to each of these connected components 
can be smoothed while keeping the intersection point fixed. Since $[f]$ was a smooth point of $M$, we conclude that $M$ parametrizes a stable map 
$f'$ whose domain has two irreducible components $C_1$ and $C_2$ and $\deg f'|_{C_i}=e_i$, $i=1,2$.

If $e_2 \geq 2$, then applying the induction hypothesis to $e_1$, we see that $f'|_{C_1}$ can be deformed to the union of a free line and a free curve $D$ of degree $e_1-1$ interesting it, and $f'|_{C_2}$ can be deformed to a free curve $D'$ of degree $e_2$ intersecting $D$.  Smoothing $D \cup D'$, we conclude that $M$ contains a map 
parametrized by  $M_{1}^0 \times_X M_{e-1}^0$. Since $M_{1}^0 \times_X M_{e-1}^0$ is irreducible and a general point of it is contained in the smooth locus of $\overline{M}_{0,0}(X,e)$, we conclude that the component $M$ is unique. 

\end{proof}

\newcommand{\closer}{\vspace{-1.5ex}}

\end{document}